\pdfoutput=1
\documentclass[12pt,reqno,a4paper]{amsart}

\usepackage{iftex}
\ifPDFTeX
  \usepackage[utf8]{inputenc}
  \usepackage[T1]{fontenc}
\fi
\usepackage[english]{babel}
\usepackage[a4paper,margin=1.12in]{geometry}

\usepackage{microtype}
\usepackage{mathpazo}
\usepackage{euler}

\usepackage{amsmath,amssymb,amsthm,mathtools,mathrsfs,bm,eucal,tensor}
\usepackage{stmaryrd}
\usepackage{dsfont}
\usepackage{aliascnt}
\numberwithin{equation}{section}

\usepackage{graphicx}
\usepackage[dvipsnames,x11names]{xcolor}
\usepackage{tabularx}
\usepackage{booktabs}
\usepackage{makecell}
\usepackage{siunitx}

\usepackage[all,cmtip]{xy}
\usepackage{tikz}
\usetikzlibrary{
  arrows.meta,
  calc,
  positioning,
  fit,
  backgrounds,
  shapes.geometric,
  decorations.markings,
  matrix,
  shapes,
  decorations.pathmorphing
}
\usepackage{tikz-cd}

\tikzset{
  >=Stealth,
  box/.style={draw, rounded corners, inner sep=6pt},
  v/.style={draw, circle, inner sep=1.4pt},
  arr/.style={->, thick},
  darr/.style={->, thick, dashed},
  midarrow/.style={
    postaction={decorate},
    decoration={markings,mark=at position 0.58 with {\arrow{Stealth}}}
  },
  vertex/.style={circle,draw,inner sep=2pt},
  active/.style={thick},
  inactive/.style={dashed}
}
\tikzcdset{arrow style=tikz, diagrams={>=Stealth}}

\usepackage{enumitem}
\usepackage{comment,braket,xspace,csquotes}
\usepackage{indentfirst}
\usepackage{relsize}
\usepackage[tablename=Table]{caption}
\usepackage[active]{srcltx}

\theoremstyle{plain}
\newtheorem{theorem}{Theorem}[section]

\newaliascnt{lemma}{theorem}
\newtheorem{lemma}[lemma]{Lemma}
\aliascntresetthe{lemma}

\newaliascnt{proposition}{theorem}
\newtheorem{proposition}[proposition]{Proposition}
\aliascntresetthe{proposition}

\newaliascnt{corollary}{theorem}
\newtheorem{corollary}[corollary]{Corollary}
\aliascntresetthe{corollary}

\theoremstyle{definition}

\newaliascnt{definition}{theorem}
\newtheorem{definition}[definition]{Definition}
\aliascntresetthe{definition}

\newaliascnt{notation}{theorem}

\aliascntresetthe{notation}

\newaliascnt{remark}{theorem}
\newtheorem{remark}[remark]{Remark}
\aliascntresetthe{remark}

\newaliascnt{assumption}{theorem}

\aliascntresetthe{assumption}

\newaliascnt{example}{theorem}
\newtheorem{example}[example]{Example}
\aliascntresetthe{example}

\newaliascnt{conjecture}{theorem}

\aliascntresetthe{conjecture}

\usepackage{hyperref}
\hypersetup{
  colorlinks=true,
  linkcolor=blue!60!black,
  citecolor=blue!60!black,
  urlcolor=blue!60!black
}
\usepackage[capitalize,nameinlink]{cleveref}

\Crefname{theorem}{Theorem}{Theorems}
\Crefname{lemma}{Lemma}{Lemmas}
\Crefname{proposition}{Proposition}{Propositions}
\Crefname{corollary}{Corollary}{Corollaries}
\Crefname{definition}{Definition}{Definitions}
\Crefname{notation}{Notation}{Notations}
\Crefname{remark}{Remark}{Remarks}
\Crefname{assumption}{Assumption}{Assumptions}
\Crefname{example}{Example}{Examples}
\Crefname{conjecture}{Conjecture}{Conjectures}

\DeclareMathOperator{\Sing}{Sing}

\DeclareMathOperator{\rk}{rk}

\DeclareMathOperator{\Res}{Res}

\DeclareMathOperator{\Var}{Var}

\newcommand{\cO}{\mathcal{O}}
\newcommand{\cF}{\mathcal{F}}

\newcommand{\Der}{\operatorname{Der}}

\newcommand{\ba}{\begin{array}}
\newcommand{\ea}{\end{array}}

\date{}

\begin{document}

\title{Excess logarithmic residues for foliations by curves and applications}

\author[A. Cavalcante]{Alana Cavalcante}
\address[A. Cavalcante]{Departamento de Matem\'atica, Universidade Federal de Vi\c cosa, Av. P. H. Rolfs, s/n, Vi\c cosa, Brazil}
\email{alana@ufv.br}

\author[M. Corr\^ea]{Maur\'icio Corr\^ea}
\address[M. Corr\^ea]{Dipartimento di Matematica, Universit\`a degli Studi di Bari, Via E. Orabona 4, I-70125 Bari, Italy}
\email{mauricio.correa.mat@gmail.com, mauricio.barros@uniba.it}

\author[F. Louren\c co]{Fernando Louren\c co}
\address[F. Louren\c co]{DMM, Universidade Federal de Lavras, Campus Universit\'ario, Lavras, Brazil}
\email{fernando.lourenco@ufla.br}

\author[E. Shahsavaripour]{Elaheh Shahsavaripour}
\address[E. Shahsavaripour]{Universit\`a degli Studi di Bari, Via E. Orabona 4, I-70125 Bari, Italy}
\email{e.shahsavaripour@gmail.com}



 \begin{abstract}
 We introduce excess logarithmic residues for one-dimensional holomorphic
foliations tangent to a divisor. They arise from the comparison between the
logarithmic normal sheaf and the ordinary normal sheaf of the foliation, and
measure the local variation between the logarithmic and classical Baum--Bott
contributions. We prove a global residue formula expressing the corresponding
Chern numbers as sums of local residues. We then derive a Poincar\'e-type bound
for invariant hypersurfaces from the non-negativity of the relevant logarithmic
residues. Finally, for a normal \(\mathbb Q\)-Gorenstein surface \(Y\), we show
that the componentwise logarithmic residues of a lifted foliation along the
exceptional divisor of a functorial resolution recover the log discrepancies of
the singularities of \(Y\), giving a dynamical and foliated test for log
canonicity of these singularities.
\end{abstract}

\maketitle

\section{Introduction}

Let \(\cF\) be a holomorphic foliation of codimension \(q\) on a complex
manifold \(X\), and let \(U=X\setminus \Sing(\cF)\) be its regular locus. On
\(U\), the normal bundle \(N_{\cF}|_U\) carries the Bott partial connection in
the directions tangent to the leaves. Bott's vanishing theorem implies that
every characteristic form associated with a homogeneous polynomial of degree
greater than \(q\) vanishes on \(U\). Nevertheless, the corresponding
characteristic class of the normal sheaf may be non-zero on \(X\). Baum and
Bott proved that such classes are represented by residues supported on the
singular set of the foliation; see \cite{BB2,BB1}. In the case of foliations by
curves with isolated singularities, these local contributions are given by
explicit Grothendieck residues; see also \cite{Suwa}.

This paper develops an excess logarithmic residue theory for foliations by
curves tangent to a divisor. Let \(D\subset X\) be a reduced divisor. Following
Saito, the logarithmic tangent sheaf is
\[
T_X(-\log D)=\Der_X(\log D),
\]
the sheaf of holomorphic vector fields preserving the ideal of \(D\); see
\cite{Saito}. A one-dimensional foliation \(\cF\) is logarithmic along \(D\)
when
\[
T_{\cF}\subset T_X(-\log D).
\]
Thus the local vector fields defining \(\cF\) are tangent to \(D\). We assume
throughout that \(T_{\cF}\) is saturated in \(T_X(-\log D)\), so that the
logarithmic normal sheaf below is torsion-free. The same foliation then has two
natural transverse sheaves: the ordinary normal sheaf
\[
N_{\cF}:=T_X/T_{\cF},
\]
and the logarithmic normal sheaf
\[
N_{\cF/X}^{\log}:=T_X(-\log D)/T_{\cF}.
\]
The inclusion \(T_X(-\log D)\subset T_X\) induces a canonical morphism
\[
N_{\cF/X}^{\log}\longrightarrow N_{\cF}.
\]
The residues considered here are attached to this morphism. They measure the
local change in the Baum--Bott contribution when the ordinary transverse
geometry of \(X\) is replaced by the logarithmic transverse geometry determined
by \(D\). In this sense, the word \textit{excess} refers to the local defect produced
by the comparison
\[
N_{\cF/X}^{\log}\longrightarrow N_{\cF}.
\]

Let us spell out this comparison locally. If \(D\) is defined near a point by a
reduced equation \(f=0\), then a vector field \(v\) is logarithmic along \(D\)
precisely when \(v(f)\in(f)\), or equivalently when \(v(f)/f\) is holomorphic.
Hence the quotient \(T_X/T_X(-\log D)\) is supported on \(D\) and is governed by
the Jacobian data of the divisor. Thus replacing \(T_X\) by \(T_X(-\log D)\)
does not merely impose tangency to \(D\); it changes the transverse sheaf in
which the foliation is measured.

The main local invariant introduced in this paper is the excess logarithmic
residue
\[
\widehat{\operatorname{Res}}^{\log}_{c_1^{n-i},\,D^i}(\cF,D,p),
\]
defined by localising the Chern--Weil representative of the class
\[
c_1\bigl(N_{\cF/X}^{\log}\bigr)^{n-i}D^i.
\]
Here, for simplicity, we write \(D\) also for the cohomology class
\([D]=c_1(\cO_X(D))\). In coordinates adapted to the divisor at a point of
\(D_{\mathrm{reg}}\), this residue admits a Grothendieck-residue expression
relative to \(D\), in the sense of Brasselet--Seade--Suwa
\cite[\S~6.3]{BrSeSu}. More precisely, if \(D=\{f=0\}\) and
\[
v=
\sum_{j=1}^{n-1}a_j\frac{\partial}{\partial z_j}
+
kf\frac{\partial}{\partial f},
\qquad k=\frac{v(f)}{f},
\]
then, writing \(\bar a_j=a_j|_D\) and denoting by \(J_D(v)\) the Jacobian
matrix of the vector field induced by \(v\) on \(D\), one obtains, for
\(1\leq i\leq n-1\),
\[
\widehat{\operatorname{Res}}^{\log}_{c_1^{n-i},\,D^i}(\cF,D,p)
=
\operatorname{Res}_{p}
\left[
\frac{
\bigl(\operatorname{tr}J_D(v)\bigr)^{n-i}
\left(\left. v(f)/f\right|_D\right)^{i-1}
\,dz_1\wedge\cdots\wedge dz_{n-1}
}
{\bar a_1,\ldots,\bar a_{n-1}}
\right]_{D}.
\]

For comparison, we also consider the ordinary local residue
\[
\operatorname{Res}_{c_1^{n-i},\,D^i}(\cF,D,p),
\]
obtained by localising the corresponding class built from the ordinary normal
sheaf \(N_{\cF}\). The excess between the ordinary and logarithmic transverse
geometries is measured by the variational residue
\[
\operatorname{Var}_{c_1^{n-i},\,D^i}(\cF,D,p).
\]
It is defined by an explicit transgression form, and we prove that
\[
\operatorname{Var}_{c_1^{n-i},\,D^i}(\cF,D,p)
=
\operatorname{Res}_{c_1^{n-i},\,D^i}(\cF,D,p)
-
\widehat{\operatorname{Res}}^{\log}_{c_1^{n-i},\,D^i}(\cF,D,p).
\]
Thus the variational residue is the excess term measuring the passage from the
logarithmic normal sheaf to the ordinary normal sheaf. Precise definitions and
the local Grothendieck-residue formulae are given in
Section~\ref{sec:residues}.

\begin{theorem}[=Theorem~\ref{thm:residues}]
Let \(\cF\) be a one-dimensional holomorphic foliation on \(X\) with only
isolated singularities, logarithmic along \(D\). Let \(Z_D(\cF)\) denote the
finite set of isolated zeroes of the vector field induced by \(\cF\) on
\(D_{\mathrm{reg}}\). When these zeroes coincide with
\(\Sing(\cF)\cap D_{\mathrm{reg}}\), we write the sums simply over
\(\Sing(\cF)\cap D\).
For \(i=0\), one has the logarithmic formula
\[
\sum_{p\in \Sing(\cF)}
\widehat{\operatorname{Res}}^{\log}_{c_1^n}(\cF,D,p)
=
\int_X c_1\bigl(N^{\log}_{\cF/X}\bigr)^n,
\]
and the variational identity
\[
\sum_{p\in \Sing(\cF)}
\operatorname{Var}_{c_1^n}(\cF,D,p)
=
\int_X
\Bigl[
c_1(N_{\cF})^n
-
c_1\bigl(N^{\log}_{\cF/X}\bigr)^n
\Bigr].
\]
For every \(1\leq i\leq n-1\), one has
\[
\sum_{p\in Z_D(\cF)}
\widehat{\operatorname{Res}}^{\log}_{c_1^{n-i},\,D^i}
(\cF,D,p)
=
\int_X c_1\bigl(N^{\log}_{\cF/X}\bigr)^{n-i}D^i,
\]
and
\[
\sum_{p\in Z_D(\cF)}
\operatorname{Var}_{c_1^{n-i},\,D^i}
(\cF,D,p)
=
\int_X
\Bigl[
c_1(N_{\cF})^{n-i}
-
c_1\bigl(N^{\log}_{\cF/X}\bigr)^{n-i}
\Bigr]D^i.
\]
\end{theorem}

We apply this formula to the Poincar\'e problem for foliations by curves on
projective space. Let \(\cF\) be a one-dimensional foliation on
\(\mathbb P^n\), and let \(D\subset\mathbb P^n\) be an invariant hypersurface
of degree \(m\). If \(\deg\cF=d\), the Poincar\'e problem asks for bounds on
\(m\) in terms of \(d\). This question has been studied, in different forms, by
Soares, Carnicer, Campillo--Carnicer--de la Fuente Garc\'ia, Brunella--Mendes
and others; see \cite{BM,CC,Carnicer,CJ,CoMa2,Esteves,So,So0}.

The logarithmic residues introduced here give a   criterion for such a
bound. More precisely, set
\[
i=
\begin{cases}
0, & \text{if } n \text{ is odd},\\
1, & \text{if } n \text{ is even}.
\end{cases}
\]
If the total logarithmic residue is non-negative, namely
\[
\sum_p
\widehat{\operatorname{Res}}^{\log}_{c_1^{n-i},\,D^i}
(\cF,D,p)\geq 0,
\]
then the global residue formula gives
\[
m\leq d+n.
\]
In particular, this hypothesis is satisfied if all the corresponding local
logarithmic residues are non-negative.

The local algebra behind the construction is closely related to Aleksandrov's
work on logarithmic differential forms, torsion differentials and
multidimensional residue theory; see \cite{Alex1,Ale_s}. For vector fields
tangent to hypersurfaces, the logarithmic index compares the usual local index
with the index computed in the module of logarithmic derivations. This is the
local model for the phenomenon studied here: the ambient module in which the
foliation is measured is changed, and the corresponding local residue changes
accordingly.

Global logarithmic residue formulae for foliations were obtained in
\cite{CoMa1,CoMa3}. More recently, logarithmic Bott localisation formulae for
global sections of \(T_X(-\log D)\) with non-isolated local complete
intersection zero schemes were established in \cite{CaEla}. The logarithmic
Baum--Bott theory for foliations by curves was further developed in
\cite{CoLoMa}. The present paper refines this circle of ideas by isolating the
local excess term associated with the comparison
\[
\begin{tikzcd}
N_{\cF/X}^{\log} \arrow[r] & N_{\cF}.
\end{tikzcd}
\]
It is therefore neither simply the classical Baum--Bott residue nor merely its
logarithmic analogue, but the residue measuring the passage from the
logarithmic normal sheaf to the ordinary normal sheaf.

Finally, in Section~\ref{sec:exceptional-residues}, we prove a birational
application in dimension two. Theorem~\ref{thm:exceptional-residues-discrepancies}
shows that, for a normal \(\mathbb Q\)-Gorenstein surface, the exceptional
logarithmic residues of a lifted foliation on a functorial log resolution
recover the log discrepancy vector. If
\[
K_X=\pi^*K_Y+\sum_j a_jD_j
\]
and \(M=(D_j\cdot D_k)\) is the exceptional intersection matrix, then
\[
(a_1+1,\ldots,a_r+1)^t=-M^{-1}I,
\]
where
\[
I=(I_1,\ldots,I_r)^t,
\qquad
I_j=
\sum_{p\in Z_{D_j}(\widetilde{\cF})}
\widehat{\operatorname{Res}}^{\log}_{c_1,D_j}
(\widetilde{\cF},D,p).
\]
Thus, under the isolatedness assumptions of
Theorem~\ref{thm:exceptional-residues-discrepancies}, the residue vector gives a
dynamical and foliated test for log canonicity.

The paper is organised as follows. In Section~\ref{sec:preliminaries} we recall
the basic material on holomorphic foliations, logarithmic vector fields and
logarithmic normal sheaves. In Section~\ref{sec:residues} we introduce the
local logarithmic residue and the associated ordinary and variational residues.
In Section~\ref{sec:global-formula} we prove the global residue formula. In
Section~\ref{sec:poincare} we apply the formula to invariant hypersurfaces of
one-dimensional foliations on projective space and derive the announced
Poincar\'e-type bound. Finally, in Section~\ref{sec:exceptional-residues} we
show that, for singular surfaces, exceptional logarithmic residues on a
functorial log resolution recover the log discrepancy vector.

\section{Preliminaries}\label{sec:preliminaries}

\begin{definition}
Let $X$ be a connected complex manifold of dimension $n$. A holomorphic
distribution of codimension $k$ on $X$ is a saturated coherent subsheaf
\[
T_{\cF}\subset T_X
\]
of rank $n-k$. Thus the quotient
$
N_{\cF}:=T_X/T_{\cF}
$
is torsion-free; it is called the normal sheaf of the distribution.
The distribution is integrable, and hence defines a holomorphic foliation, if
$T_{\cF}$ is closed under the Lie bracket: for every open set $U\subset X$ and
all sections $\xi,\eta\in H^0(U,T_{\cF})$, one has
$
[\xi,\eta]\in H^0(U,T_{\cF}).
$
In this case one has the exact sequence
\[
0\longrightarrow T_{\cF}
\longrightarrow T_X
\longrightarrow N_{\cF}
\longrightarrow 0,
\]
and the singular locus of $\cF$ is
\[
\Sing(\cF):=\Sing(N_{\cF})
=
\{p\in X;\; N_{\cF,p}\ \text{is not locally free}\}.
\]
\end{definition}

The inclusion of the conormal sheaf $
N_{\cF}^{\vee}\hookrightarrow \Omega_X^1
$
determines a twisted holomorphic $k$-form
$$
\omega_{\cF}\in H^0\bigl(X,\Omega_X^k\otimes \det(N_{\cF})\bigr).
$$
Locally one may write
\[
\omega_{\cF}=\omega_1\wedge\cdots\wedge\omega_k,
\]
and the integrability condition is the Frobenius condition
\[
d\omega_i\wedge \omega_1\wedge\cdots\wedge\omega_k=0,
\qquad i=1,\ldots,k.
\]

When $k=n-1$, the foliation has dimension one, or equivalently is a foliation
by curves. In this case $T_{\cF}$ has rank one, and $\cF$ may be represented by
a twisted holomorphic vector field
\[
v\in H^0(X,T_X\otimes \mathcal L),
\]
for a suitable line bundle $\mathcal L$. On $X=\mathbb P^n$, a
one-dimensional foliation of degree $d$ is represented by such a section with
$
\mathcal L=\mathcal O_{\mathbb P^n}(d-1),
$
with $d\geq 0$. 
\begin{definition}
Let $D$ be a divisor on $X$, locally given by $D=\{f=0\}$. A holomorphic vector
field $v$ is logarithmic along $D$, or tangent to $D$, if it preserves the ideal
of $D$; equivalently,
$
v(f)=hf
$
for some holomorphic function $h$. The sheaf of logarithmic vector fields along
$D$ is denoted by
\[
T_X(-\log D)=\Der_X(\log D).
\]
\end{definition}
Thus $T_X(-\log D)$ consists   of those holomorphic vector fields whose
local flow preserves $D$. This sheaf is fundamental in the theory of
logarithmic forms and logarithmic derivations; see, for instance,
\cite{Saito,Ale_s,Alex1,Alu11}.
We shall consider one-dimensional holomorphic foliations which are logarithmic
along $D$, namely foliations satisfying
$
T_{\cF}\subset T_X(-\log D).
$
This inclusion yields the logarithmic tangent exact sequence
\[
0\longrightarrow T_{\cF}
\longrightarrow T_X(-\log D)
\longrightarrow N_{\cF/X}^{\log}
\longrightarrow 0,
\]
where, throughout the paper, we assume that \(T_{\cF}\) is saturated in
\(T_X(-\log D)\) and set
\[
N_{\cF/X}^{\log}:=T_X(-\log D)/T_{\cF}.
\]
Thus \(N_{\cF/X}^{\log}\) is torsion-free; when it is locally free, it is the
logarithmic normal bundle of the foliation. Its determinant is the line bundle
\[
\det N_{\cF/X}^{\log}:=
\bigl(\wedge^{n-1}N_{\cF/X}^{\log}\bigr)^{\ast\ast},
\]
and we set
\[
c_1(N_{\cF/X}^{\log}):=c_1(\det N_{\cF/X}^{\log}).
\]

The sheaf $N_{\cF/X}^{\log}$ describes the transverse geometry of $\cF$ after
the divisor has been incorporated into the tangent structure of $X$. It has
appeared in recent work on logarithmic foliations and logarithmic residue
formulae; see, for example, \cite{CoLoMa,CoMa1,CoMa3,CaEla}.

\subsection{Chern--Weil theory of characteristic classes}

Although \(N_{\cF}\) and \(N_{\cF/X}^{\log}\) need not be locally free everywhere,
only their first Chern classes and powers of these classes are used below. For a
torsion-free coherent sheaf \(\mathcal E\) on the smooth manifold \(X\), we set
\[
c_1(\mathcal E):=c_1(\det\mathcal E),
\qquad
\det\mathcal E=
\bigl(\wedge^{\rk\mathcal E}\mathcal E\bigr)^{\ast\ast}.
\]
Thus \(c_1(\mathcal E)\) is the first Chern class of a genuine line bundle, and
all powers \(c_1(\mathcal E)^k\) used below are understood in this sense.

\begin{definition} A connection for a complex vector bundle $E$ on $M$ is a $\mathbb{C}$-linear map
$$ \nabla : A^{0} (X,E) \longrightarrow A^{1} (X, E) $$
that satisfies
$$ \nabla(fs)=df\otimes s+f\nabla(s) \ \ \mbox{for} \ \ f \in A^{0}(X) \ \ \mbox{and} \ \ s \in A^{0}(X,E).$$

\end{definition}
If $\nabla$ is a connection for $E$, then it induces a $\mathbb{C}$-linear map
$$ \nabla := \nabla^{2} : A^{1}(X,E) \longrightarrow A^{2}(X,E) $$

\noindent satisfying
$$ \nabla (\omega \otimes s) = d\omega \otimes s - \omega \wedge \nabla (s), \ \ \omega \in A^{1}(X), \ \ s \in A^{0} (X,E).$$

\begin{definition} The composition $K := \nabla \circ \nabla : A^{0} (X,E) \longrightarrow A^{2} (X,E)$ is called the curvature of the connection $\nabla$.

\end{definition}

If $\nabla$ denotes a connection for a vector bundle $E$ of rank $r$ and $E$ is trivial on the open set $U$, i.e., $E|_{U} \simeq U \times \mathbb{C}^{r}$ and if $ s = ( s_{1},\ldots,s_{r})$ is a frame of $E$ on $U$, then we can write
$$ \nabla (s_{i}) = \sum_{j = 1}^{r} \theta_{ij} \otimes s_{j} \ \ ; \ \ \theta_{ij} \in A^{1}(U). $$
The connection matrix with respect to $s$ is $\theta = (\theta_{ij})$ . Also, using the curvature definition, we get
$$ K(s_{i}) = \sum_{j = 1}^{r}K_{ij}s_{j}, \ \ \ \mbox{where} \ \ \ K_{ij} = d \theta_{ij} - \sum_{k =1}^{r} \theta_{ik} \wedge \theta_{kj}.$$
The curvature matrix with respect to the frame $s$ is $K = (K_{ij})$. Now, to define the Chern class of a vector bundle $E$, we consider $\sigma_{i}, i = 1,\ldots,r$ the $i$-th elementary symmetric functions in the eigenvalues of the matrix $K$
$$ \det(It + K) = 1 + \sigma_{1}(K)t + \sigma_{2}(K)t^{2} + \cdots + \sigma_{r}(K)t^{r}. $$
Next, we define a $2i$-form of Chern $c_{i}$ on $U$ by
$$ c_{i}(K) := \sigma_{i} \Big(\frac{{   i}}{2 \pi}K\Big). $$

In general, if $\varphi$ is a symmetric polynomial  in $r$ variables of degree $d$, we can write $\varphi = \tilde{P}(c_{1},\ldots,c_{r})$ for some polynomial $\tilde{P}$. Then we can define
$$ \varphi(K) := \tilde{P}(c_{1}(K),\ldots,c_{r}(K))$$
\noindent which is a closed form on $X$. Therefore, we have a cohomology class of $E$ on $X$, $\varphi (E) := \varphi(K) \in H^{2d}(X; \mathbb{C})$. 
We shall use the following elementary facts. Let $E$ be a complex vector bundle
of rank $r<n$, and let $\nabla$ be any connection on $E$. Since the Chern forms
are obtained from an $r\times r$ curvature matrix, one has
$
c_p(\nabla)=0 \qquad \text{for all } p>r.
$
Consequently, if
$
\varphi=c_{p_1}\cdots c_{p_k}
$
is a Chern monomial and some $p_i>r$, then
$
\varphi(\nabla)=0.$
Moreover, on any open set $U\subset X$ over which $E$ is trivial, with local
frame $(s_1,\ldots,s_r)$, one may choose the flat local connection defined by
\[
\nabla s_i=0,\qquad i=1,\ldots,r.
\]
Its curvature matrix is identically zero; hence all positive Chern forms vanish
on $U$. In particular,
\[
c_i(E)|_U=0\in H^{2i}(U,\mathbb C),\qquad i>0.
\]
Let  $E_i$, $i=0,\ldots,q$ be a family of complex vector bundle on a complex manifold $M$, $ \xi$ be the virtual bundle $\displaystyle \xi=\sum _{i=0} ^{q}(-1)^i E_i$ and $\varphi$ be a  homogeneous symmetric polynomial. By the standard definition of characteristic classes of virtual bundles, we have $$\varphi(\xi)=\sum _{l} \varphi_l ^{(0)} (E_0)\varphi_l ^{(1)}(E_1)\cdots\varphi_l ^{(q)}(E_q),$$ where $\varphi_l ^{(i)}(E_i)$ is a polynomial in the Chern classes of $E_i$, for each $i$ and $l$.
If $\nabla^{(i)}$ is a connection for $E_i$ consider the family $\nabla^{\bullet}=(\nabla^{(q)},\ldots,\nabla^{(0)}).$ Then $\varphi(\xi)$ is the cohomology class of the differential form 
\begin{equation} \varphi(\nabla^{\bullet})=\sum _{l} \varphi_l ^{(0)} (\nabla^{(0)} )\wedge\varphi_l ^{(1)}(\nabla^{(1)})\wedge\cdots\wedge\varphi_l ^{(q)}(\nabla^{(q)}).\end{equation}
For further details, we refer the reader to \cite{Suwa}.

Given a holomorphic foliation $\cF$ of codimension $k$ on a complex manifold $X,$ with its short exact sequence
\[
\begin{tikzcd}
0 \arrow[r] & T_{\cF} \arrow[r] & T_X \arrow[r] & N_{\cF} \arrow[r] & 0
\end{tikzcd}
\]
we have a very well defined element $\varphi(N_{\cF}) \in H^{2d}(U;\mathbb{C})$, where $d = \deg(\varphi).$ From the commutative diagram with $S$ a compact connected component of the singular set of $\cF$
\[
\begin{tikzcd}
H^{2d}(X,U; \mathbb{C}) \arrow[r] \arrow[d, "A"'] &
H^{2d}(X; \mathbb{C}) \arrow[d, "P"] \\
H_{2n-2d}(S; \mathbb{C}) \arrow[r, "i_{\ast}"'] &
H_{2n-2d}(X; \mathbb{C})
\end{tikzcd}
\]
 we have the Baum-Bott Theorem \cite{BB1}
$$P(\varphi(N_{\cF})) = \sum_{ S} i_{\ast}\Res_{\varphi}(\cF, S).$$
If $S$ is an isolated singular point, and $\cF$ of a one-dimensional holomorphic foliation given in a local chart by the vector field 
$v = \sum_{i=1}^{n} a_i \partial /\partial z_i$
 then the above expression became 
$$\int_{X} \varphi(N_{\cF}) = \Res_{\varphi}(\cF, S),$$

\noindent where the residue is a Grothendieck residue defined by

$$\Res_{\varphi}(\cF, S) = \Res_{p} \left[
\begin{array}{cc}
\varphi(Jv)  \\
{a_1,\dots,a_{n}}
\end{array}\right] := \dfrac{1}{(2\pi {   i})^{n}}\int_{\Gamma} \dfrac{\varphi(Jv)}{a_1 \ldots a_n}, $$

\noindent where $\Gamma$ is the corresponding real $n$-cycle. Moreover, if we consider \(\cF\) logarithmic along the reduced divisor
\(D=\{f=0\}\), then \(v(f)=kf\) for some holomorphic function \(k\).
On the smooth locus \(D_{\mathrm{reg}}\) one has the usual tangent-normal
sequence
\[
\begin{tikzcd}
0 \arrow[r] &
T_{D_{\mathrm{reg}}} \arrow[r] &
T_X|_{D_{\mathrm{reg}}} \arrow[r] &
N_{D_{\mathrm{reg}}/X} \arrow[r] &
0 .
\end{tikzcd}
\]
In general, we  have  the tangent sheaf
$
T_D=\mathcal Hom_{\mathcal O_D}(\Omega_D^1,\mathcal O_D).
$
The natural morphism
$$
T_X|_D\longrightarrow \mathcal O_D(D)$$
has image \(J_D(D)\), where \(J_D\subset\mathcal O_D\) is the Jacobian ideal of
\(D\). Thus one has the exact Jacobian-image sequence
\[
\begin{tikzcd}
0 \arrow[r] &
T_D \arrow[r] &
T_X|_D \arrow[r] &
J_D(D) \arrow[r] &
0 .
\end{tikzcd}
\]
Equivalently, viewing the quotient as an \(\mathcal O_X\)-sheaf supported on
\(D\), this is encoded by the exact sequence
\[
0\longrightarrow T_X(-\log D)
\longrightarrow T_X
\longrightarrow J_D(D)
\longrightarrow 0.
\]
  Now, for a homogeneous symmetric polynomial $\varphi$ of degree $n-1$ and the virtual tangent bundle on $D$ defined by $\tau := TX|_{D}- N_{D}$, we write $\varphi = \sum_{l}\varphi_{l}.\varphi_{l}^{'}$ so that $\varphi(\tau) = \sum_{l}\varphi_{l}(TX).\varphi_{l}^{'}(N_{D})$ and set $\varphi(H) = \sum_{l}\varphi_{l}(Jv).\varphi_{l}^{'}(k)$. And we have the following residue formula
\begin{align*}
\Res_{\varphi}(\cF,\tau,p)
&=\Res_{p}
\left[
\begin{array}{c}
\varphi(H)\,dz_{1}\wedge \cdots \wedge dz_{n-1}  \\
{\bar{a}}_1,\ldots,\bar{a}_{n-1}
\end{array}
\right]_{D} \\[15pt]
&=\Res_{p}
\left[
\begin{array}{c}
\varphi(H)\,dz_{1}\wedge \cdots \wedge dz_{n-1}\wedge df  \\
a_1,\ldots,a_{n-1},f
\end{array}
\right],
\end{align*}

\noindent where $\bar{a}_i = a_i|_{D}.$

\subsection{Logarithmic Holomorphic Foliations}
Let $X$ be a compact complex manifold of dimension $n \geq 3$, let $D=\{f=0\} \subset X$ be a reduced divisor, and let $\cF$ be a one-dimensional holomorphic foliation  logarithmic along $D$ with singular set of codimension at least two and induced by the vector field $v \in H^0(X, T_X \otimes \mathcal{L}^{\ast}).$ We have the short exact sequence 
\[
\begin{tikzcd}
0 \arrow[r] &
T_{\cF} \arrow[r] &
T_X(-\log D) \arrow[r] &
N_{\cF/X}^{\log} \arrow[r] &
0
\end{tikzcd}
\]
where
\[
T_X(-\log D):=\ker\bigl(T_X\xrightarrow{\ \phi_f\ }\cO_D(D)\bigr)
\]
is the logarithmic tangent sheaf. We assume that \(T_{\cF}\) is saturated in
\(T_X(-\log D)\) and define
\[
N_{\cF/X}^{\log}:=T_X(-\log D)/T_{\cF}.
\]
Thus \(N_{\cF/X}^{\log}\) is torsion-free.
Moreover, we have the exact sequence
\[
0\longrightarrow T_X(-\log D)
\longrightarrow T_X
\longrightarrow J_D(D)
\longrightarrow 0,
\]
where $J_D\subset \cO_D$ denotes the Jacobian ideal of $D$ on $D$, and
$J_D(D)$ is regarded as a coherent $\cO_X$-sheaf supported on $D$.


%

 \begin{remark}
Let \(U := X \setminus (\Sing(\cF) \cup \Sing D)\).
On \(U\), both \(T_X(-\log D)_{|U}\) and $$T_X(-\log D)/T_{\cF}$$ are holomorphic vector bundles/quotients. All differential-form computations are done on this locally free locus and are interpreted in cohomology via the determinant line bundles. The explicit Grothendieck-residue formulae in adapted coordinates are stated at points of \(D_{\mathrm{reg}}\); at points of \(\Sing D\) one has to use the relative residue theory on the singular hypersurface \(D\).
 \end{remark}
 
\begin{definition} 
Let $f \in \cO_{X,x}$ be a reduced local equation of $D$ at $x \in X.$ With local coordinates $z_1, \ldots, z_n,$ and $\partial_i=\partial/\partial_{z_i},$
\begin{center}
$ (J_D)_x := (\overline{\partial_1 f}, \ldots, \overline{\partial_n f}) \subset \cO_{D,x} := \cO_{X,x}/(f).$
\end{center}
\end{definition}
If $f' = uf$ with $u \in \cO_{X,x}^{\times},$ then $\overline{\partial_i f'} = \overline{u} \overline{\partial_i f},$so $J_D$ is well-defined on $D.$

\begin{proposition} \label{prop:exact-sequence}
Fix a local reduced equation $f$ of $D.$ Define
$$\phi_f:T_X \to \cO_D(D), \;\;\; \phi_f(\xi) = \Big(\frac{\xi(f)}{f} \Big)_{|D} s_D.$$
Then $\ker(\phi_f) = T_X(-\log D)$ and $Im(\phi_f) = J_D(D).$ Hence there is a canonical short exact sequence 
\begin{equation} \label{eq:exact-sequence}
0 \to T_X(-\log D) \to T_X \xrightarrow{\phi_f} J_D (D) \to 0.
\end{equation}
\end{proposition}
\begin{proof}
Write $\xi = \sum a_i \partial_i,$ then $\xi(f) = \sum a_i \partial_i f.$ By definition, $\xi \in T_X(-\log D)$ iff $\xi(f) \in (f),$ i.e. $(\xi(f)/f)|_D = 0.$ Thus $\ker \phi_f = T_X(-\log D).$ The image consists of linear combinations $\sum \overline{a_i}\overline{\partial_i f},$ hence equals $J_D \cdot s_D = J_D(D).$
If $f'=uf$ with $u \in \cO_{X,x}^{\times},$ then
$$\frac{\xi(f')}{f'} = \frac{\xi(uf)}{uf} = \frac{\xi(f)}{f} + \frac{\xi(u)}{u}.$$
The term $\xi(u)/u$ is holomorphic along $D,$ hence defines a section of $\cO_X \subset \cO_X(D)$ and maps to zero in the quotient $\cO_D(D) = \cO_X(D)/\cO_X.$ Therefore, $\phi_{f'} = \phi_f$ as morphisms $T_X \to \cO_D(D),$ and both kernel and image subsheaves are independent of the choice of $f.$
\end{proof}

Suppose that $D$ is a free divisor at $p,$ by Saito’s criterion \cite{Saito}, there exist $n$ vector fields germs $\delta_1, \ldots,\delta_n  \in H^0(T_{X,p}(-\log D))$ such that $\gamma = \{\delta_1, \ldots ,\delta_n\}$ is a system of $\cO_{X,p}$-free basis for $T_{X,p}(-\log D)$ and
$\delta_j = \sum_{i=1}^n a_{ij} \partial /\partial z_i$
is such that $\det(a_{ij}) = g \cdot f,$ where $g$ is a unit and $D = \{f = 0\}.$  Moreover, $$[\delta_i, \delta_j] =  \displaystyle\sum_{k=1}^n \delta_{ij}^k \delta_k,$$ where $\delta_{ij}^k \in \cO_p,$ for all $i,j,k \in \{1, \ldots ,n\}.$
The basis $\{\delta_1, \ldots, \delta_n\}$ is called Saito’s basis of $T_{X,p}(-\log D))$ at $p.$

\begin{corollary} 
 Let $x \in X.$ If $D$ has normal crossings at $x$ with local equation $z_1 \ldots z_r=0$ in coordinates $(z_1, \ldots, z_r, w_{r+1}, \ldots, w_n),$ then 
  \(T_X(-\log D)_x\) is generated, as an \(\cO_{X,x}\)-module, by
\(z_1\frac{\partial}{\partial z_1},\ldots,
z_r\frac{\partial}{\partial z_r},
\frac{\partial}{\partial w_{r+1}},\ldots,
\frac{\partial}{\partial w_n}\).
 In general  , $T_X(-\log D) = \ker \phi$ by Proposition \ref{prop:exact-sequence}.
\end{corollary}

\begin{proof}
$(z_{\alpha} \partial{z_{\alpha}})(z_1 \ldots z_r) = z_1 \ldots z_r \in (f)$ and $\partial_{w_j}(f)=0 \in (f).$ Conversely, any vector field preserving $(f)$ is an $\cO_X$-combination of these generators.    
\end{proof}

\begin{proposition}  
The one-dimensional holomorphic foliation sequence is
\begin{equation} \label{eq:fol-sequence}
0 \to T_{\cF} \to T_X \to N_{\cF} \to 0,
\end{equation}
where $N_{\cF}:=T_X/T_{\cF}$ is torsion-free of rank $n-1.$ Comparing (\ref{eq:exact-sequence}) and (\ref{eq:fol-sequence}) with $T_{\cF} \subset T_X(-\log D)$ yields the exact sequence 
\begin{equation} \label{eq:normal-comparison}
0 \to T_X(- \log D)/T_{\cF} \to N_{\cF} \to J_D(D) \to 0.
\end{equation}
\end{proposition}

\begin{proof}
Applying the Snake Lemma to the diagram given by $T_{\cF} \subset T_X(- \log D) \subset T_X$ and the quotient $T_X \to J_D(D).$ 
\end{proof}


We have the   diagram

 \[
\begin{tikzcd}[row sep=large, column sep=large]
& 0 \arrow[d] & 0 \arrow[d] & 0 \arrow[d] & \\
0 \arrow[r] &
T_{\cF} \arrow[r] \arrow[d] &
T_X \arrow[r] \arrow[d] &
N_{\cF} \arrow[r] \arrow[d] &
0 \\
0 \arrow[r] &
T_X(-\log D) \arrow[r] \arrow[d] &
T_X \arrow[r] \arrow[d] &
J_D(D) \arrow[r] \arrow[d] &
0 \\
0 \arrow[r] &
N_{\cF/X}^{\log} \arrow[r] \arrow[d] &
N_{\cF} \arrow[r] \arrow[d] &
J_D(D) \arrow[r] \arrow[d] &
0 \\
& 0 & 0 & 0 &
\end{tikzcd}
\]
Since \(T_{\cF}\) is saturated in \(T_X(-\log D)\), the left-hand quotient in
\eqref{eq:normal-comparison} is precisely \(N^{\log}_{\cF/X}\). Taking determinants
in \eqref{eq:normal-comparison} gives
\[
\det N_{\cF}
\simeq
\det N^{\log}_{\cF/X}\otimes \det J_D(D)
\]
in codimension one. Consequently,
\[
c_1(N_{\cF})
=
c_1\bigl(N^{\log}_{\cF/X}\bigr)
+
c_1\bigl(J_D(D)\bigr)
=
c_1\bigl(N^{\log}_{\cF/X}\bigr)+D,
\]
since \(c_1\bigl(J_D(D)\bigr)=D\). We shall only use this first-Chern-class
identity, not an unqualified identity of total Chern classes for the coherent
sheaves involved.

\subsection{Pfaff systems, the trace form, and change of frame}

Motivated by the flag formalism for holomorphic foliations \cite{BCL}, applied here to the
  inclusions
\[
T_{\cF}\subset T_X(-\log D)\subset T_X,
\]
we work locally on the regular locus of \(\cF\) and choose Pfaff data adapted to
this pair of transverse structures. Thus there exists a decomposable
\((n-1)\)-form
\[
\omega_{12}=\eta_1\wedge\cdots\wedge\eta_{n-1}
\]
such that
\[
\omega_1:=\omega_2\wedge\omega_{12}
=
h\cdot \eta_1\wedge\eta_2\wedge\cdots\wedge\eta_{n-1}
\]
induces the foliation \(\cF\).
From the integrability condition we can write
\[
d\eta_u=\sum_{v=1}^{n-1}\theta_{uv}\wedge\eta_v,
\qquad 1\leq u\leq n-1,
\]
or
\[
\begin{array}{cccc}
\left\{
\begin{array}{llccc}
d\eta_1
=
\theta_{11}\wedge\eta_1
+
\theta_{12}\wedge\eta_2
+\cdots+
\theta_{1\,n-1}\wedge\eta_{n-1}, \\[4pt]
d\eta_2
=
\theta_{21}\wedge\eta_1
+
\theta_{22}\wedge\eta_2
+\cdots+
\theta_{2\,n-1}\wedge\eta_{n-1}, \\
\vdots \\
d\eta_{n-1}
=
\theta_{n-1\,1}\wedge\eta_1
+
\theta_{n-1\,2}\wedge\eta_2
+\cdots+
\theta_{n-1\,n-1}\wedge\eta_{n-1}.
\end{array}
\right.
\end{array}
\]
Then, we define the trace form
\begin{equation}\label{eq:theta12}
\theta^{12}:=\operatorname{tr}(\theta)
=
\sum_{u=1}^{n-1}\theta_{uu}.
\end{equation}

Now, passing from the ordinary transverse structure to the logarithmic one, we
use the integrability condition of \(\cF\) induced by
\[
\omega_1
=
\widetilde{\eta}_1\wedge\eta_2\wedge\cdots\wedge\eta_{n-1},
\]
where \(\widetilde{\eta}_1=h\eta_1\). Thus
\[
\begin{array}{cccc}
\left\{
\begin{array}{llccc}
d\widetilde{\eta}_1
=
\widetilde{\theta}_{11}\wedge\widetilde{\eta}_1
+
\widetilde{\theta}_{12}\wedge\eta_2
+\cdots+
\widetilde{\theta}_{1\,n-1}\wedge\eta_{n-1}, \\[4pt]
d\eta_2
=
\widetilde{\theta}_{21}\wedge\widetilde{\eta}_1
+
\widetilde{\theta}_{22}\wedge\eta_2
+\cdots+
\widetilde{\theta}_{2\,n-1}\wedge\eta_{n-1}, \\
\vdots \\
d\eta_{n-1}
=
\widetilde{\theta}_{n-1\,1}\wedge\widetilde{\eta}_1
+
\widetilde{\theta}_{n-1\,2}\wedge\eta_2
+\cdots+
\widetilde{\theta}_{n-1\,n-1}\wedge\eta_{n-1}.
\end{array}
\right.
\end{array}
\]
We define again
\begin{equation}\label{eq:theta12til}
\widetilde{\theta}^{12}
:=
\operatorname{tr}(\widetilde\theta)
=
\sum_{u=1}^{n-1}\widetilde\theta_{uu}.
\end{equation}

\begin{lemma}  \label{lem:pfaff-frame}
On overlaps, if $\eta' = g \cdot \eta$ with $g:U \to GL_{n-1}(\mathbb{C})$ holomorphic, then $\theta'_{uv} = (g^{-1} \theta g + g^{-1}dg)_{uv}$ and $\theta^{12'} = \theta^{12} + d \log \det g.$ Thus the classes and boundary integrals defined below are independent of the Pfaff data.  
\end{lemma}
\begin{proof}
Taking the trace of \(\theta'=g^{-1}\theta g+g^{-1}dg\) gives
\[
\operatorname{tr}(\theta')
=
\operatorname{tr}(\theta)+\operatorname{tr}(g^{-1}dg)
=
\operatorname{tr}(\theta)+d\log\det g.
\]
\end{proof}

\begin{lemma}\label{lem:JD-local-potential}
Let \(D\) also denote the cohomology class \(c_1(\mathcal O_X(D))\in
H^2(X,\mathbb C)\). Choose a smooth closed \(2\)-form
 $
F_D\in A^2(X)
$
representing \(D\). On every sufficiently small coordinate open set
\(V\subset X\) there exists a \(1\)-form \(\beta\in A^1(V)\) such that
$
d\beta=F_D|_V.
$
Moreover, if \(\beta'\) is another local primitive of \(F_D|_V\), then, after
shrinking \(V\) if necessary, there exists a smooth function \(\rho\) on \(V\)
such that
$
\beta'-\beta=d\rho.
$
Consequently, for every pair of closed forms \(B,C\) on \(V\), the forms
\[
\beta'\wedge B^r\wedge C^s
\quad\text{and}\quad
\beta\wedge B^r\wedge C^s
\]
differ by an exact form. In particular, the boundary integrals below are
independent of the chosen local primitive \(\beta\).
\end{lemma}

\begin{proof}
The existence of a smooth closed representative \(F_D\) of the class
\(D=c_1(\mathcal O_X(D))\) is standard de Rham theory. Since \(F_D\) is closed,
the Poincar\'e lemma gives, on every sufficiently small coordinate open set
\(V\), a \(1\)-form \(\beta\) satisfying \(d\beta=F_D|_V\).
If \(\beta'\) is another such primitive, then
\[
d(\beta'-\beta)=0.
\]
After shrinking \(V\) if necessary, the Poincar\'e lemma again gives a smooth
function \(\rho\) such that
$\beta'-\beta=d\rho.
$
Let \(B\) and \(C\) be closed forms on \(V\). Then
\[
\beta'\wedge B^r\wedge C^s
-
\beta\wedge B^r\wedge C^s
=
d\rho\wedge B^r\wedge C^s.
\]
Since \(dB=dC=0\), this is
\[
d\bigl(\rho\, B^r\wedge C^s\bigr).
\]
Thus the two forms differ by an exact form. Hence their integrals over closed
cycles, in particular over the links used below, are equal.
\end{proof}
 
\section{ Logarithmic gauge, transgressions, and their differentials}

Let $f$ be a local reduced equation of $D$ on $U.$ Passing from the block $T_{\cF} \subset T_X$ to $T_{\cF} \subset T_X(- \log D)$ gives the {\it logarithmic gauge}
\begin{equation} \label{eq:log-gauge}
\theta^{12} = \tilde{\theta}^{12} - d \log h, \;\;\;\; d\theta^{12} = d\tilde{\theta}^{12}:= F_{12}. 
\end{equation}
In an adapted chart along \(D_{\mathrm{reg}}\) one may take \(h=uf\), with \(u\) a unit;
therefore \(d\log h=d\log f+d\log u\), and the term \(d\log u\) contributes only an
exact transgression to the boundary integrals.

Fix once and for all a closed $2$-form $F_2$ representing $D$ (Lemma \ref{lem:JD-local-potential}), locally write $F_2= d \beta.$
For $0 \leq i \leq n-1$ define the $(2n-1)$- forms
\begin{equation}
    \psi_i := (2 \pi i)^{-n} \theta^{12} \wedge F_2^i \wedge F^{n-1-i}_{12},
\end{equation}

\begin{equation}
    \psi_i^{\log} := (2 \pi i )^{-n} \tilde{\theta}^{12} \wedge F_2^i \wedge F^{n-1-i}_{12},
\end{equation}

\begin{lemma}  \label{lem:exact-difference}
On $U$,
\begin{equation} \label{eq:phi-difference}
\phi_i:= \psi_i -  \psi_i^{\log} = -(2 \pi {   i})^{-n} d \log h \wedge F_2^i \wedge F^{n-1-i}_{12}.
\end{equation}
If $h'=uh$ with $u$ a unit, the right-hand side changes by an exact form; its integral over boundaries vanishes. 
\end{lemma}

\begin{proof}
Insert (\ref{eq:log-gauge}) into (\ref{eq:phi-difference}); use $d(d \log h)=0$ and that $F_2, F_{12}$ are closed.
\end{proof}

\begin{lemma}   \label{lem:differentials}
On $U,$
$
d \psi_i = d \psi_i^{\log} = (2 \pi { i})^{-n} F^{n-i}_{12} \wedge F_2^i.
$
\end{lemma}

\begin{proof}
Let $A:= \theta^{12}$ or $\tilde{\theta}^{12}, B:=F_2$ and $C:=F_{12}.$ Then 
$$\psi_i = (2 \pi {   i})^{-n} A \wedge B^i \wedge C^{n-1-i},$$ and since $dB = dC = 0,$ we obtain 
$$d \psi_i = (2 \pi {   i})^{-n} dA \wedge B^i \wedge C^{n-1-i} = (2 \pi {   i})^{-n} C^{n-i} \wedge B^i.$$
The same computation applies to $\psi^{\log}_i.$
\end{proof}

\section{ Chern-Weil identification of the right-hand side}

Let \(F_{12}=d\tilde\theta^{12}\) be the curvature representative of
\(c_1(N^{\log}_{\cF/X})\) associated with the logarithmic Bott connection on the
locally free locus. Equivalently, it represents the first Chern class of the
line bundle \(\det N^{\log}_{\cF/X}\) and extends as a closed current/form on \(X\)
across subsets of codimension at least two.
Then,
\begin{equation}
    \Omega_i:= (2 \pi {   i})^{-n} F^{n-i}_{12} \wedge F_2^i
\end{equation}
represents 
$$c_1(N^{\log}_{\cF/X})^{n-i} D^i \in H^{2n}(X, \mathbb{C}), $$
and by Lemma \ref{lem:differentials} we have $d \psi_i^{\log} = \Omega_i$ on $U.$

\section{Links, local residues and Stokes}\label{sec:residues}\label{sec:global-formula}

For later use, set \(\Sigma_0:=\Sing(\cF)\). For \(i\geq 1\), let
\(\Sigma_i:=Z_D(\cF)\) denote the finite set of isolated zeroes of the vector field
induced by \(\cF\) on \(D_{\mathrm{reg}}\). When these zeroes coincide with
\(\Sing(\cF)\cap D_{\mathrm{reg}}\), we simply write the sums over
\(\Sing(\cF)\cap D\).

\subsection{Links $L_p$ and orientations}

Fix a Hermitian metric on \(X\). For each \(p\in\Sigma_i\), pick \(\epsilon_p>0\) such that the closed balls
\(B_p := \{x \in X: \mbox{dist}(x,p) \leq \epsilon_p\}\)
are disjoint and contained in charts where \(\cF\) is regular off \(p\) and, if
\(p\in D_{\mathrm{reg}}\), where \(D=\{f=0\}\). Define the link 
$$L_p:= \partial B_p \simeq S^{2n-1},$$
with boundary orientation (outward normal first). Then $L_p \subset U.$
Using any local potential $\beta$ with $d \beta = F_2,$ set

\begin{equation} \label{eq:local-log-residue}
    \widehat{\mbox{Res}}^{log}_{c_1^{n-i}, D^i}(\cF,D,p) = (2 \pi {   i})^{-n} \int_{L_p} \tilde{\theta}^{12} \wedge (d \beta)^i \wedge F_{12}^{n-1-i} := \int_{L_p} \psi_i^{\log}.
\end{equation}

\begin{equation} 
    \mbox{Res}_{c_1^{n-i}, D^i}(\cF,D,p) :=  \int_{L_p} \psi_i  .
\end{equation}

\begin{equation} 
    \Var_{c_1^{n-i},\,D^i}(\cF,D,p)
:=\int_{L_p}\phi_i.
\end{equation}

By Lemmas \ref{lem:pfaff-frame}, \ref{lem:JD-local-potential} and \ref{lem:exact-difference} this is independent of all auxiliary choices $(\eta, F_2, \beta, h).$

\subsection{Stokes on the complement and integrability}

Let $$M_i:=X \setminus \bigcup_{p \in \Sigma_i} \mbox{int}(B_p).$$
Then \(\partial M_i = \bigsqcup_{p\in\Sigma_i} L_p\). Since \(d \psi_i^{\log} = \Omega_i\) on the relevant locally free locus,
\begin{equation}
    \int_{M_i} \Omega_i = \int_{M_i} d \psi_i^{\log} = \int_{\partial M_i} \psi_i^{\log} = \sum_{p\in\Sigma_i} \int_{L_p} \psi_i^{\log}.
\end{equation}

\begin{lemma}\label{lem:integrability}
Let
$
M_i:=X\setminus\bigcup_{p\in\Sigma_i}\operatorname{int}(B_p).$
Then
\[
\lim_{\epsilon_p\to 0}
\left(
\int_X\Omega_i-\int_{M_i}\Omega_i
\right)
=0.
\]
Equivalently,
\[
\lim_{\epsilon_p\to 0}
\sum_{p\in\Sigma_i}\int_{B_p}\Omega_i=0.
\]
\end{lemma}

\begin{proof}
It is enough to prove that \(\Omega_i\) is locally integrable near each point of
\(\Sigma_i\). On the locally free locus of the logarithmic normal sheaf, the form
\[
\Omega_i=(2\pi i)^{-n}F_{12}^{\,n-i}\wedge F_D^i
\]
is smooth away from \(D\). Along \(D_{\mathrm{reg}}\) it has at worst logarithmic
growth, because it is built from logarithmic connection forms and their curvature
forms. Hence \(\Omega_i\) defines a locally integrable top-degree form near the
points under consideration.
Therefore the integral of \(\Omega_i\) over a ball \(B_p\) tends to zero when the
radius of \(B_p\) tends to zero:
\[
\lim_{\epsilon_p\to0}\int_{B_p}\Omega_i=0.
\]
Since \(\Sigma_i\) is finite, summing over \(p\in\Sigma_i\) gives
\[
\lim_{\epsilon_p\to0}
\sum_{p\in\Sigma_i}\int_{B_p}\Omega_i=0.
\]
The result follows from
\[
X=M_i\cup\bigcup_{p\in\Sigma_i}B_p
\]
up to boundaries of measure zero.
\end{proof}

We are now in a position to prove the main result of the paper.

\begin{theorem}\label{thm:residues}
Let \(\cF\) be a one-dimensional holomorphic foliation on \(X\) with only
isolated singularities, logarithmic along \(D\). Let \(Z_D(\cF)\) denote the
finite set of isolated zeroes of the vector field induced by \(\cF\) on
\(D_{\mathrm{reg}}\). When these zeroes coincide with
\(\Sing(\cF)\cap D_{\mathrm{reg}}\), we write the sums simply over
\(\Sing(\cF)\cap D\).
For \(i=0\), one has the logarithmic formula
\[
\sum_{p\in \Sing(\cF)}
\widehat{\operatorname{Res}}^{\log}_{c_1^n}(\cF,D,p)
=
\int_X c_1\bigl(N^{\log}_{\cF/X}\bigr)^n,
\]
and the variational identity
\[
\sum_{p\in \Sing(\cF)}
\operatorname{Var}_{c_1^n}(\cF,D,p)
=
\int_X
\Bigl[
c_1(N_{\cF})^n
-
c_1\bigl(N^{\log}_{\cF/X}\bigr)^n
\Bigr].
\]
For every \(1\leq i\leq n-1\), one has
\[
\sum_{p\in Z_D(\cF)}
\widehat{\operatorname{Res}}^{\log}_{c_1^{n-i},\,D^i}
(\cF,D,p)
=
\int_X c_1\bigl(N^{\log}_{\cF/X}\bigr)^{n-i}D^i,
\]
and
\[
\sum_{p\in Z_D(\cF)}
\operatorname{Var}_{c_1^{n-i},\,D^i}
(\cF,D,p)
=
\int_X
\Bigl[
c_1(N_{\cF})^{n-i}
-
c_1\bigl(N^{\log}_{\cF/X}\bigr)^{n-i}
\Bigr]D^i.
\]
\end{theorem}

\begin{proof}
We prove the logarithmic identities. The ordinary identities are obtained in
the same way by replacing the logarithmic normal sheaf by the ordinary normal
sheaf, and the variational identities follow by subtraction.
First consider \(i=0\). Let
\[
V_0=\bigcup_{p\in \Sing(\cF)}B_p
\]
be a union of pairwise disjoint sufficiently small balls centred at the isolated
singularities of \(\cF\), and put \(L_p=\partial B_p\). On the punctured ball
\(B_p^\ast=B_p\setminus\{p\}\), choose Pfaff data as in the preceding
construction:
\[
\omega_{12}=\eta_1\wedge\cdots\wedge\eta_{n-1},
\qquad
\omega_1=\omega_2\wedge\omega_{12}
=
h\,\eta_1\wedge\cdots\wedge\eta_{n-1},
\]
with \(\omega_1\) defining \(\cF\). The integrability relations give the trace
forms \(\theta^{12}\) and \(\widetilde\theta^{12}\), and we set
\[
F_{12}^{\log}:=d\widetilde\theta^{12}.
\]
The form
\[
\psi_0^{\log}
:=
(2\pi i)^{-n}
\widetilde\theta^{12}\wedge
\bigl(F_{12}^{\log}\bigr)^{n-1}
\]
satisfies
\[
d\psi_0^{\log}
=
(2\pi i)^{-n}
\bigl(F_{12}^{\log}\bigr)^n,
\]
which represents \(c_1(N^{\log}_{\cF/X})^n\). Since \(\cF\) has dimension one,
its codimension is \(n-1\), and Bott's vanishing theorem localizes this
degree-\(n\) characteristic class on \(\Sing(\cF)\). Therefore Stokes' theorem
on \(X\setminus V_0\), followed by the limit as the radii of the balls tend to
zero, gives
\[
\sum_{p\in \Sing(\cF)}
\widehat{\operatorname{Res}}^{\log}_{c_1^n}(\cF,D,p)
=
\sum_{p\in \Sing(\cF)}
\int_{L_p}\psi_0^{\log}
=
\int_X c_1\bigl(N^{\log}_{\cF/X}\bigr)^n.
\]

Now let \(1\leq i\leq n-1\). In this case the factor \(D^i\) makes the
localization relative to the divisor. On \(D_{\mathrm{reg}}\), the logarithmic
vector field induces a one-dimensional foliation of codimension \(n-2\). The
relative characteristic polynomial has degree \(n-1\), and Bott's vanishing
theorem on \(D_{\mathrm{reg}}\) localizes the corresponding relative class at
the isolated zeroes of the induced vector field, namely at \(Z_D(\cF)\).
Let
\[
V_D=\bigcup_{p\in Z_D(\cF)}B_p
\]
be a union of pairwise disjoint sufficiently small balls centred at these
points. On each punctured ball choose the same local Pfaff data as above and
also a local primitive \(\beta\) of the fixed closed representative \(F_D\) of
the class \(D\), so that \(d\beta=F_D\). Set
\[
\psi_i^{\log}
:=
(2\pi i)^{-n}
\widetilde\theta^{12}\wedge
F_D^i\wedge
\bigl(F_{12}^{\log}\bigr)^{n-i-1}.
\]
Since \(dF_D=dF_{12}^{\log}=0\), one has
\[
d\psi_i^{\log}
=
(2\pi i)^{-n}
F_D^i\wedge
\bigl(F_{12}^{\log}\bigr)^{n-i}.
\]
This form represents
\[
c_1\bigl(N_{\cF/X}^{\log}\bigr)^{n-i}D^i.
\]
By the relative Bott localization described above, this class is represented
by a current supported on \(\overline V_D\). Applying Stokes' theorem on
\(X\setminus V_D\) and then letting the radii of the balls tend to zero, we get
\[
\sum_{p\in Z_D(\cF)}
\widehat{\operatorname{Res}}^{\log}_{c_1^{n-i},\,D^i}(\cF,D,p)
=
\sum_{p\in Z_D(\cF)}
\int_{L_p}\psi_i^{\log}
=
\int_X
c_1\bigl(N_{\cF/X}^{\log}\bigr)^{n-i}D^i.
\]
The proof of the ordinary identities is identical, replacing
\(\widetilde\theta^{12}\) and \(F_{12}^{\log}\) by the corresponding ordinary
trace form and curvature representative of \(c_1(N_{\cF})\). Finally, by
definition,
\[
\operatorname{Var}_{c_1^{n-i},D^i}(\cF,D,p)
=
\operatorname{Res}_{c_1^{n-i},D^i}(\cF,D,p)
-
\widehat{\operatorname{Res}}^{\log}_{c_1^{n-i},D^i}(\cF,D,p).
\]
Thus the variational identities follow by subtracting the logarithmic
identities from the ordinary ones.
\end{proof}
 Using Lemma \ref{lem:integrability} and (\ref{eq:local-log-residue}) we get

\begin{corollary}
Under the assumptions of Theorem \ref{thm:residues}, we recover for $i=0,$

\begin{itemize}
\item[(a)] The Baum-Bott Theorem for logarithmic foliation in \cite[Theorem~1.2]{CoLoMa}. 
\item[(b)] The classical Baum-Bott Theorem, if $D= \emptyset$.
\end{itemize}
\end{corollary}

\section{Applications to the Poincar\'e problem}\label{sec:poincare}

We begin with the surface case. The logarithmic residue formula recovers
Brunella's interpretation of the GSV index and, consequently, Carnicer's bound
for nondicritical invariant curves in \(\mathbb P^2\).

\begin{theorem}[Brunella \cite{brunella}]\label{thm:poincare-plane} Let $\cF$ be a one-dimensional holomorphic foliation on $X=\mathbb{P}^{2}$ logarithmic along $D$. If $D$ is a nondicritical separatrix at singular points $p$ of $\cF$, then

$$\deg D \leq \deg \cF +2.$$

\noindent Moreover, the equality holds if and only if $\cF$ is a generalized curve.
\end{theorem}

\begin{proof}
Indeed, applying Theorem~\ref{thm:residues} \textup{(a)} to \(X=\mathbb P^2\) and
taking \(i=1\), we obtain
\[
\begin{aligned}
\sum_{p \in \Sing(\cF)\cap D}
\widehat{\operatorname{Res}}^{\log}_{c_1,\,D}(\cF,D,p)
&=
\int_D c_1\bigl(N^{\log}_{\cF/X}\bigr)\,c_1\bigl(J_D(D)\bigr) \\
&=
(\deg \cF+2-\deg D)\deg D.
\end{aligned}
\]
We now use the short exact sequence
\[
\begin{tikzcd}
0 \arrow[r] &
N_{\cF/X}^{\log} \arrow[r] &
N_{\cF} \arrow[r] &
J_D(D) \arrow[r] &
0 .
\end{tikzcd}
\]
It gives
$
c_1\bigl(N_{\cF/X}^{\log}\bigr)
=
c_1\bigl(N_{\cF}\bigr)-c_1\bigl(J_D(D)\bigr).
$
Since, for brevity, we write \(D=c_1\bigl(J_D(D)\bigr)\), this becomes
$
c_1\bigl(N_{\cF/X}^{\log}\bigr)
=
c_1\bigl(N_{\cF}\bigr)-D.
$
Intersecting with \(D\), we get
$
c_1\bigl(N_{\cF/X}^{\log}\bigr)D
=
\bigl(c_1(N_{\cF})-D\bigr)D.
$
By \cite[Proposition~4, p.~532]{brunella}, the logarithmic residue along the
invariant curve is identified with the GSV index:
\[
\sum_{p \in \Sing(\cF)\cap D}
\widehat{\operatorname{Res}}^{\log}_{c_1,\,D}(\cF,D,p)
=
\sum_{p \in \Sing(\cF)\cap D}
\operatorname{GSV}(\cF,D,p).
\]
Since \(D\) is a nondicritical separatrix, one has
$
\operatorname{GSV}(\cF,D,p)\geq 0
$
for every \(p\in \Sing(\cF)\cap D\). The desired inequality therefore follows.
Finally, the equality case is characterised by the vanishing of all the local
GSV indices. By \cite[Th\'eor\`eme~3.3]{CaLeh}, this is equivalent to saying 
that the singularities of \(\cF\) along \(D\) are generalized curves.
\end{proof}

We observe that Theorem \ref{thm:poincare-plane} recoverss Carnicer's result in \cite[Theorem p. 289]{Carnicer}. Now, let $X$ be a complex compact surface, and $D \subset X$ a reduced curve, by Brunella \cite{brunella} we have
$$\sum_{p \in \Sing(\cF)\cap D} GSV(\cF, D,p) =(c_1(N_{\cF}) - D)D.$$
and
$$\sum_{p \in \Sing(\cF)\cap D} CS(\cF, D,p) =D^{2}.$$
Moreover, $\Var(\cF,D,p) = GSV(\cF,D,p)+CS(\cF,D,p). $
Then Theorem \ref{thm:residues}\textup{(b)}, for $n=2$ and $i=1$, gives
$$\displaystyle \sum_{p \in \Sing(\cF)\cap D}{\mbox{Res}}_{c_1, D}(\cF,D,p)= \displaystyle \sum_{p \in \Sing(\cF)\cap D}\Var(\cF,D,p).$$
Similarly, Theorem \ref{thm:residues}\textup{(c)}, for $n=2$ and $i=1$, gives
$$
\begin{array}{ccl}
\displaystyle \sum_{p \in \Sing(\cF)\cap D}\Var_{c_1,\,D}(\cF,D,p)&=& \Big(c_{1}(N_{\cF}) - c_{1}(N_{\cF/X}^{\log})\Big)D \\
&=& \Big(c_{1}(N_{\cF}) - c_{1}(N_{\cF}) + D\Big) . D \\ \\
&=& D^{2} \\ \\
&=& \displaystyle \sum_{p \in \Sing(\cF)\cap D} CS(\cF, D,p).
\end{array}
$$

Finally, the following result shows that the non-negativity of the total logarithmic residue yields a positive answer to the Poincar\'e problem.
\begin{corollary}
Let \(X=\mathbb P^n\), let \(\deg \cF=d\), and let
\(D\subset\mathbb P^n\) be an invariant hypersurface of degree \(m\). Set
\[
i=
\begin{cases}
0, & \text{if } n \text{ is odd},\\
1, & \text{if } n \text{ is even}.
\end{cases}
\]
Assume that the total logarithmic residue is non-negative, namely
\[
\sum_p
\widehat{\operatorname{Res}}^{\log}_{c_1^{n-i},\,D^i}
(\cF,D,p)\geq 0.
\]
Then
$
m\leq d+n.
$
\end{corollary}

\begin{proof}
Let \(H\) denote the hyperplane class. Since
\[
c_1\bigl(N^{\log}_{\cF/X}\bigr)=(d+n-m)H,
\qquad
c_1\bigl(J_D(D)\bigr)=mH,
\]
Theorem~\ref{thm:residues}\textup{(a)} gives
\[
\sum_{p\in\Sigma_i}
\widehat{\operatorname{Res}}^{\log}_{c_1^{n-i},\,D^i}
(\cF,D,p)
=
(d+n-m)^{n-i}m^i.
\]
By the choice of \(i\), the integer \(n-i\) is odd. The left-hand side is the
Chern number \(\int_{\mathbb P^n}c_1(N^{\log}_{\cF/X})^{n-i}D^i\), hence an
integer; by assumption it is non-negative. As \(m>0\), it follows that
\(d+n-m\geq 0\), hence
$
m\leq d+n.
$
\end{proof}

\section{Determination and computation of logarithmic residues}

We now give local expressions for the residues introduced above, with special
attention to the variational correction term.

\begin{definition} \label{def:var-residue}
For \(0\leq i\leq n-1\) and \(p\in \Sing(\cF)\cap D\), define
\[
\Var_{c_1^{n-i},\,D^i}(\cF,D,p)
:=
\int_{L_p}\phi_i,
\]
where \(\phi_i\) is the exact-difference form of
Lemma~\ref{lem:exact-difference}.
\end{definition}

\begin{proposition}\label{prop:var-circle}
Fix \(p\in \Sing(\cF)\cap D\), and keep the local data of the preceding
sections. Let \(f\) be a reduced local equation of \(D\) on a chart \(U\)
containing \(L_p\), and choose a local potential \(\beta\) such that
\(d\beta=F_D\), where \(F_D\) represents \(D\). Set \(K_p:=L_p\cap D\). On
\(U\setminus D\), define
\[
\alpha_i
:=
\frac{1}{n-1}\Bigl(
(n-1-i)\,\widetilde{\theta}^{12}\wedge F_D^i\wedge F_{12}^{n-2-i}
+
i\,\beta\wedge F_D^{i-1}\wedge F_{12}^{n-1-i}
\Bigr),
\]
with the convention that the second summand is omitted if \(i=0\), and the
first summand is omitted if \(i=n-1\). Then:
\begin{enumerate}
\item[\textup{(a)}] On \(U\setminus D\) one has
\[
d\alpha_i=F_D^i\wedge F_{12}^{n-1-i},
\]
and therefore
\[
\phi_i=(2\pi i)^{-n}\,d\bigl(d\log f\wedge \alpha_i\bigr).
\]

\item[\textup{(b)}] Let \(T_\varepsilon\) be a sufficiently small closed tubular
neighbourhood of \(K_p\) in \(L_p\), and put
\[
L_{p,\varepsilon}:=L_p\setminus \operatorname{Int}(T_\varepsilon).
\]
Then the limit \(\lim_{\varepsilon\to 0}\int_{L_{p,\varepsilon}}\phi_i\)
exists, and
\[
\Var_{c_1^{n-i},\,D^i}(\cF,D,p)
=
(2\pi i)^{-n+1}\int_{K_p}\alpha_i,
\]
where \(K_p=\partial(D\cap B_p)\) is endowed with the induced boundary
orientation.
\end{enumerate}
In particular, if \(n=2\) and \(i=0\), then
\(\alpha_0=\widetilde{\theta}^{12}\) and
\[
\Var_{c_1^2,\,D^0}(\cF,D,p)
=
\frac{1}{2\pi i}\int_{\partial(D\cap B_p)}
\widetilde{\theta}^{12}.
\]
\end{proposition}

\begin{proof}
By Lemma~\ref{lem:differentials}, together with \(dF_D=dF_{12}=0\), we have
\[
d\bigl(\widetilde{\theta}^{12}\wedge F_D^i\wedge F_{12}^{n-2-i}\bigr)
=
F_D^i\wedge F_{12}^{n-1-i},
\]
and similarly
\[
d\bigl(\beta\wedge F_D^{i-1}\wedge F_{12}^{n-1-i}\bigr)
=
F_D^i\wedge F_{12}^{n-1-i}.
\]
Since the coefficients in the definition of \(\alpha_i\) add up to one, it
follows that \(d\alpha_i=F_D^i\wedge F_{12}^{n-1-i}\). Moreover, since
\(d(d\log f)=0\) on \(U\setminus D\),
\[
d\bigl(d\log f\wedge\alpha_i\bigr)
=
-\,d\log f\wedge d\alpha_i
=
-\,d\log f\wedge F_D^i\wedge F_{12}^{n-1-i}.
\]
By Lemma~\ref{lem:exact-difference}, this is the exact-difference representative
giving \(\phi_i\), with the stated normalisation. This proves \textup{(a)}.

For \textup{(b)}, apply Stokes' theorem on \(L_{p,\varepsilon}\). With the
above orientation conventions,
\[
\int_{L_{p,\varepsilon}}\phi_i
=
(2\pi i)^{-n}
\int_{\partial T_\varepsilon}d\log f\wedge\alpha_i.
\]
On \(\partial T_\varepsilon\), the map \(f/|f|\) gives the circle fibration by
positively oriented small loops around \(D\), and
$
\int_{S^1}d\log f=2\pi i
$
on each fibre. Hence
\[
\int_{\partial T_\varepsilon}d\log f\wedge\alpha_i
=
(2\pi i)\int_{K_p}\alpha_i.
\]
Thus
\[
\int_{L_{p,\varepsilon}}\phi_i
=
(2\pi i)^{-n+1}\int_{K_p}\alpha_i.
\]
Letting \(\varepsilon\to 0\) and using Definition~\ref{def:var-residue} gives
the claim.
\end{proof}

\begin{corollary} \label{cor:global-var}
In the setting of Theorem~\ref{thm:residues}, for every \(0\leq i\leq n-1\)
one has
\[
\sum_{p\in \Sigma_i}
\Var_{c_1^{n-i},\,D^i}(\cF,D,p)
=
\int_X
\Bigl[
c_1(N_{\cF})^{n-i}
-
c_1\bigl(N^{\log}_{\cF/X}\bigr)^{n-i}
\Bigr]D^i.
\]
In particular, for \(i=0\),
\[
\sum_{p\in \Sing(\cF)}
\Var_{c_1^n,\,D^0}(\cF,D,p)
=
\int_X
\Bigl[
c_1(N_{\cF})^n
-
c_1\bigl(N^{\log}_{\cF/X}\bigr)^n
\Bigr].
\]
\end{corollary}

\begin{proof}
By Definition~\ref{def:var-residue} and by the definitions of the ordinary and
logarithmic residues, for each \(p\) one has
\[
\begin{aligned}
\Var_{c_1^{n-i},\,D^i}(\cF,D,p)
&=
\int_{L_p}\phi_i \\
&=
\int_{L_p}\psi_i-\int_{L_p}\psi_i^{\log} \\
&=
\operatorname{Res}_{c_1^{n-i},\,D^i}(\cF,D,p)
-
\widehat{\operatorname{Res}}^{\log}_{c_1^{n-i},\,D^i}(\cF,D,p).
\end{aligned}
\]
Summing over \(p\in\Sigma_i\) and applying
Theorem~\ref{thm:residues}\textup{(a)}--\textup{(b)} gives the formula. The case
\(i=0\) corresponds to \(\Sigma_0=\Sing(\cF)\).
\end{proof}

We now fix the notation used in the local Grothendieck-residue formulae. Let
\(p\in D_{\mathrm{reg}}\) be an isolated zero of the vector field induced on \(D\),
and let \(f\) be a reduced local equation of \(D\) at \(p\). Choose local coordinates
\((z_1,\ldots,z_{n-1},f)\), and write
\[
v=
\sum_{j=1}^{n-1}a_j\frac{\partial}{\partial z_j}
+
kf\frac{\partial}{\partial f},
\qquad k=\frac{v(f)}{f}.
\]
Since \(v\) is logarithmic along \(D\), the function \(k\) is holomorphic.
Set \(\bar a_j:=a_j|_D\). Let \(J_D(v)\) denote the Jacobian matrix of the
vector field induced by \(v\) on \(D\):
\[
J_D(v)=
\left(
\frac{\partial \bar a_u}{\partial z_v}
\right)_{1\leq u,v\leq n-1}.
\]
We assume below that \(p\) is an isolated zero of this induced vector field,
equivalently that
\[
(\bar a_1,\ldots,\bar a_{n-1})\subset \cO_{D,p}
\]
has zero-dimensional support.

\begin{proposition} 
\label{prop:hat-log-groth}
For \(1\leq i\leq n-1\), one has
\[
\widehat{\operatorname{Res}}^{\log}_{c_1^{n-i},\,D^i}(\cF,D,p)
=
\operatorname{Res}_{p}
\left[
\frac{
\bigl(\operatorname{tr}J_D(v)\bigr)^{n-i}
\left(\left. k\right|_D\right)^{i-1}
\,dz_1\wedge\cdots\wedge dz_{n-1}
}
{\bar a_1,\ldots,\bar a_{n-1}}
\right]_{D}.
\]
Equivalently,
\[
\widehat{\operatorname{Res}}^{\log}_{c_1^{n-i},\,D^i}(\cF,D,p)
=
\operatorname{Res}_{p}
\left[
\frac{
\bigl(\operatorname{tr}J_D(v)\bigr)^{n-i}
\left( k\right)^{i-1}
\,dz_1\wedge\cdots\wedge dz_{n-1}\wedge df
}
{a_1,\ldots,a_{n-1},f}
\right].
\]
\end{proposition}

\begin{proof}
The local logarithmic normal directions split, in the adapted frame, into the
tangential directions along \(D\) and the logarithmic normal direction generated by
\(f\partial/\partial f\). The induced vector field on \(D\) is
\[
v_D=\sum_{j=1}^{n-1}\bar a_j\frac{\partial}{\partial z_j},
\]
and its tangential linear part is the matrix \(J_D(v)\). Moreover, since
\(v(f)=kf\), the action of \(v\) on the logarithmic normal generator
\(f\partial/\partial f\) has eigenvalue \(k|_D\). Thus, after restricting the
localization of
\[
c_1\bigl(N_{\cF/X}^{\log}\bigr)^{n-i}D^i
\]
to the hypersurface \(D\), one obtains the relative characteristic polynomial
\[
\bigl(\operatorname{tr}J_D(v)\bigr)^{n-i}(k|_D)^{i-1}
\]
of total degree \(n-1\) on \(D\). By the relative Grothendieck residue formula of
Brasselet--Seade--Suwa \cite[\S 6.3]{BrSeSu}, this gives
\[
\widehat{\operatorname{Res}}^{\log}_{c_1^{n-i},\,D^i}(\cF,D,p)
=
\operatorname{Res}_{p}
\left[
\frac{
\bigl(\operatorname{tr}J_D(v)\bigr)^{n-i}
(k|_D)^{i-1}
\,dz_1\wedge\cdots\wedge dz_{n-1}
}
{\bar a_1,\ldots,\bar a_{n-1}}
\right]_{D}.
\]
This is the first displayed expression, since \(k=v(f)/f\). The ambient expression
follows from the complete-intersection identification of relative residues
\cite[\S 6.3]{BrSeSu}:
\[
\operatorname{Res}_{p}
\left[
\frac{\omega}{\bar a_1,\ldots,\bar a_{n-1}}
\right]_{D}
=
\operatorname{Res}_{p}
\left[
\frac{\omega\wedge df}{a_1,\ldots,a_{n-1},f}
\right].
\]
\end{proof}

For \(0\leq i\leq n-1\), define
\[
\Delta_i(v,D)
:=
\left(\left.\frac{v(f)}{f}\right|_D\right)^{i-1}
\left[
\left(
\operatorname{tr}J_D(v)
+
\left.\frac{v(f)}{f}\right|_D
\right)^{n-i}
-
\bigl(\operatorname{tr}J_D(v)\bigr)^{n-i}
\right],
\]
with the convention that, for \(i=0\), this means the polynomial quotient
\[
\Delta_0(v,D)
:=
\frac{
\left(
\operatorname{tr}J_D(v)
+
\left.\dfrac{v(f)}{f}\right|_D
\right)^n
-
\bigl(\operatorname{tr}J_D(v)\bigr)^n
}
{ \left.(v(f)/f)\right|_D}.
\]

\begin{proposition} 
\label{prop:var-groth}
With the notation above, for every \(0\leq i\leq n-1\), one has
\[
\Var_{c_1^{n-i},\,D^i}(\cF,D,p)
=
\operatorname{Res}_{p}
\left[
\frac{
\Delta_i(v,D)\,
dz_1\wedge\cdots\wedge dz_{n-1}
}
{\bar a_1,\ldots,\bar a_{n-1}}
\right]_{D}.
\]
\end{proposition}

\begin{proof}
By the preceding corollary, locally the variational residue is the difference
between the ordinary and logarithmic contributions:
\[
\Var_{c_1^{n-i},\,D^i}(\cF,D,p)
=
\operatorname{Res}_{c_1^{n-i},\,D^i}(\cF,D,p)
-
\widehat{\operatorname{Res}}^{\log}_{c_1^{n-i},\,D^i}(\cF,D,p).
\]
In the adapted coordinates, the ordinary linearisation has trace
\[
\operatorname{tr}J_D(v)+ \left.(v(f)/f)\right|_D,
\]
whereas the logarithmic tangential linearisation has trace
$
\operatorname{tr}J_D(v).$
After passing to the relative residue on \(D\), the normal logarithmic factor
contributes the powers of \(\left.(v(f)/f)\right|_D\) appearing in
\(\Delta_i(v,D)\). Hence the numerator of the difference is precisely
\(\Delta_i(v,D)\). The relative Grothendieck-residue formula of
Brasselet--Seade--Suwa \cite[\S 6.3]{BrSeSu} gives the stated expression.
\end{proof}

\begin{example}
Let \(\cF\) be the one-dimensional holomorphic foliation on \(\mathbb P^2\),
logarithmic along \(D=\{z_2=0\}\), induced by
\[
v
=
-3z_0\frac{\partial}{\partial z_0}
+
2z_1\frac{\partial}{\partial z_1}
+
z_2\frac{\partial}{\partial z_2}.
\]
On \(U_0=\{z_0\neq 0\}\), with coordinates \((x,y)\), this vector field becomes
\[
v_{|U_0}
=
5x\frac{\partial}{\partial x}
+
4y\frac{\partial}{\partial y},
\]
and has a unique singularity at the origin. Since \(c_1(Jv)=9\) and
\(c_1(Jv^{\log})=5\), Proposition~\ref{prop:var-groth} gives
\[
\Var_{c_1^{2-i},\,D^i}(\cF,D,p)
=
\begin{cases}
\dfrac{56}{20}, & i=0,\\[6pt]
\dfrac{16}{20}, & i=1.
\end{cases}
\]
On \(U_1=\{z_1\neq 0\}\), with coordinates \((u,t)\), the vector field becomes
\[
v_{|U_1}
=
-5u\frac{\partial}{\partial u}
-
t\frac{\partial}{\partial t},
\]
again with a unique singularity at the origin. Here \(c_1(Jv)=-6\) and
\(c_1(Jv^{\log})=-5\), and therefore
\[
\Var_{c_1^{2-i},\,D^i}(\cF,D,p)
=
\begin{cases}
\dfrac{11}{5}, & i=0,\\[6pt]
\dfrac{1}{5}, & i=1.
\end{cases}
\]
Summing the two local contributions gives
\[
\frac{56}{20}+\frac{11}{5}=5
=
c_1(N_{\cF})^2-c_1(N_{\cF}^{\log})^2
\]
for \(i=0\), and
\[
\frac{16}{20}+\frac{1}{5}=1
=
\bigl(c_1(N_{\cF})-c_1(N_{\cF}^{\log})\bigr)D
\]
for \(i=1\). This verifies Corollary~\ref{cor:global-var}, since
\(c_1(N_{\cF})=3\), \(\deg\cF=1\), \(\deg D=1\), and
\(c_1(N_{\cF}^{\log})=2\).
\end{example}

\begin{example}
Let \(X=\mathbb P^3\), with homogeneous coordinates
\([z_0:z_1:z_2:z_3]\), and consider the vector field
\[
v
=
-4z_0\frac{\partial}{\partial z_0}
+
3z_1\frac{\partial}{\partial z_1}
+
2z_2\frac{\partial}{\partial z_2}
-
z_3\frac{\partial}{\partial z_3},
\]
together with the divisor \(D=\{z_3=0\}\). Let \(\cF\) be the induced
one-dimensional foliation on \(\mathbb P^3\), which is logarithmic along \(D\).
On \(U_0\), the vector field is
\[
v_{|U_0}
=
7x\frac{\partial}{\partial x}
+
6y\frac{\partial}{\partial y}
+
3z\frac{\partial}{\partial z},
\]
so \(c_1(Jv)=16\) and \(c_1(Jv^{\log})=13\). On \(U_1\), with coordinates
\((u,q,w)\), it is
\[
v_{|U_1}
=
-7u\frac{\partial}{\partial u}
-
q\frac{\partial}{\partial q}
-
4w\frac{\partial}{\partial w},
\]
so \(c_1(Jv)=-12\) and \(c_1(Jv^{\log})=-8\). Finally, on \(U_2\), with
coordinates \((r,s,t)\), it is
\[
v_{|U_2}
=
-6r\frac{\partial}{\partial r}
+
s\frac{\partial}{\partial s}
-
3t\frac{\partial}{\partial t},
\]
so \(c_1(Jv)=-8\) and \(c_1(Jv^{\log})=-5\).
Using Proposition~\ref{prop:var-groth} and Theorem~\ref{thm:residues}, we obtain:
\begin{itemize}
\item For \(i=0\),
\[
\begin{aligned}
\sum_{p\in \Sing(\cF)\cap D}\Var_{c_1^3}(\cF,D,p)
&=
\int_{\mathbb P^3}
\Bigl[
c_1(N_{\cF})^3
-
c_1\bigl(N^{\log}_{\cF/X}\bigr)^3
\Bigr],\\
\frac{1899}{126}+\frac{1216}{28}-\frac{387}{18}
&=
4^3-3^3=37.
\end{aligned}
\]

\item For \(i=1\),
\[
\begin{aligned}
\sum_{p\in \Sing(\cF)\cap D}
\Var_{c_1^2,\,D}(\cF,D,p)
&=
\int_{\mathbb P^3}
\Bigl[
c_1(N_{\cF})^2
-
c_1\bigl(N^{\log}_{\cF/X}\bigr)^2
\Bigr]D,\\
\frac{29}{14}+\frac{80}{7}-\frac{13}{2}
&=
4^2-3^2=7.
\end{aligned}
\]

\item For \(i=2\),
\[
\begin{aligned}
\sum_{p\in \Sing(\cF)\cap D}
\Var_{c_1,\,D^2}(\cF,D,p)
&=
\int_{\mathbb P^3}
\Bigl[
c_1(N_{\cF})
-
c_1\bigl(N^{\log}_{\cF/X}\bigr)
\Bigr]D^2,\\
\frac{9}{42}+\frac{16}{7}-\frac{9}{6}
&=
4-3=1.
\end{aligned}
\]
\end{itemize}
\end{example}

\section{Exceptional residues and log discrepancies}
\label{sec:exceptional-residues}

We finish with a birational application of the logarithmic residue formula in
dimension two. The point is that the logarithmic dynamics created on a
resolution detects the discrepancies of the ambient singularity. Under the
hypotheses below, the residue vector along the exceptional components is
transformed by the inverse of the exceptional intersection matrix into the log
discrepancy vector.

Let \(Y\) be a normal compact \(\mathbb Q\)-Gorenstein complex surface with
isolated singularities, and let \(\cF\) be a one-dimensional foliation on \(Y\)
such that
\[
\Sing(\cF)\subset \Sing(Y).
\]
Thus \(\cF\) is regular on \(Y_{\mathrm{reg}}\). Let
$
\pi:(X,D)\longrightarrow (Y,\emptyset)
$
be a functorial log resolution, with
$$
D=\operatorname{Exc}(\pi)_{\mathrm{red}}
=
D_1+\cdots+D_r
$$
a simple normal crossings divisor. By functoriality of the resolution and the
invariance of \(\Sing(Y)\) under local automorphisms, vector fields defining
\(\cF\) lift to logarithmic vector fields tangent to \(D\); see
\cite[Theorems~3.34, 3.35]{kollar07} for functorial log resolutions and
\cite[Corollary~4.7]{GKK10} for the logarithmic lifting of vector fields to such
resolutions.

Assume that \(\cF\) is induced on \(Y_{\mathrm{reg}}\) by a global non-twisted
vector field. Let \(\widetilde{\cF}\) be the saturation in \(T_X(-\log D)\) of
the rank-one subsheaf generated by the lifted logarithmic vector field. We
assume that \(\widetilde{\cF}\) has no divisorial zeroes. For each component \(D_j\), let \(Z_{D_j}(\widetilde{\cF})\) denote the finite
set of logarithmic zeroes of \(\widetilde{\cF}\) which contribute to the
localization of the class
\[
c_1\bigl(N^{\log}_{\widetilde{\cF}/X}\bigr)[D_j].
\]
Define the componentwise exceptional residue
\[
I_j:=
\sum_{p\in Z_{D_j}(\widetilde{\cF})}
\widehat{\operatorname{Res}}^{\log}_{c_1,D_j}
(\widetilde{\cF},D,p).
\]
Here
\(\widehat{\operatorname{Res}}^{\log}_{c_1,D_j}\) denotes the logarithmic residue
obtained from Theorem~\ref{thm:residues} by inserting the divisor class
\([D_j]\), while the logarithmic normal sheaf is still defined with respect to
the total boundary \(D\).

Now, let
\[
I=(I_1,\ldots,I_r)^t
\]
be the residue vector, and let
\[
M=(D_j\cdot D_k)_{j,k}
\]
be the exceptional intersection matrix. Since \(Y\) is \(\mathbb Q\)-Gorenstein,
we may write, in \(N^1(X)_{\mathbb Q}\),
\[
K_X=\pi^*K_Y+\sum_{j=1}^r a_jD_j.
\]
Set
\[
b=(b_1,\ldots,b_r)^t,
\qquad
b_j:=a_j+1.
\]

\begin{theorem} 
\label{thm:exceptional-residues-discrepancies}
With the notation and assumptions above, one has
$
I=-Mb.
$
Consequently,
$
b=-M^{-1}I.
$
In particular,
\[
Y\text{ has log canonical singularities}
\quad\Longleftrightarrow\quad
(-M^{-1}I)_j\ge0, 
\quad\text{for every }j.
\]
\end{theorem}

\begin{proof}
Since the lifted foliation is generated by a global non-twisted logarithmic
vector field and has no divisorial zeroes, its saturated tangent sheaf is
$
T_{\widetilde{\cF}}\simeq \mathcal O_X.
$
Hence the logarithmic tangent sequence
\[
0\longrightarrow T_{\widetilde{\cF}}
\longrightarrow T_X(-\log D)
\longrightarrow N^{\log}_{\widetilde{\cF}/X}
\longrightarrow 0
\]
gives $
\det N^{\log}_{\widetilde{\cF}/X}
\simeq
\det T_X(-\log D).
$
Since \(D\) is a simple normal crossings divisor,
\[
\det T_X(-\log D)
\simeq
\mathcal O_X(-(K_X+D)).
\]
Therefore
$
c_1\bigl(N^{\log}_{\widetilde{\cF}/X}\bigr)=-(K_X+D).
$
By the componentwise form of Theorem~\ref{thm:residues}, obtained by inserting
the divisor class \([D_j]\), we have
\[
I_j
=
\int_X c_1\bigl(N^{\log}_{\widetilde{\cF}/X}\bigr)D_j.
\]
Thus
$
I_j=-(K_X+D)\cdot D_j.
$
Since
$
K_X=\pi^*K_Y+\sum_{k=1}^r a_kD_k $
and \(D=\sum_kD_k\), we have
\[
K_X+D
=
\pi^*K_Y+\sum_{k=1}^r(a_k+1)D_k
=
\pi^*K_Y+\sum_{k=1}^r b_kD_k.
\]
Intersecting with \(D_j\), and using
$
\pi^*K_Y\cdot D_j=0
$
because \(D_j\) is exceptional, gives
\[
(K_X+D)\cdot D_j
=
\sum_{k=1}^r b_k(D_k\cdot D_j).
\]
Therefore
\[
I_j
=
-\sum_{k=1}^r(D_k\cdot D_j)b_k.
\]
In matrix notation this is
$
I=-Mb. $
The exceptional intersection matrix of a resolution of a normal surface
singularity is negative definite, hence invertible. Therefore
\[
b=-M^{-1}I.
\]
Finally, \(Y\) is log canonical if and only if
$
a_j\ge -1
\quad\text{for every }j, $
equivalently if and only if
$
b_j=a_j+1\ge0
\quad\text{for every }j.
$
This proves the last assertion.
\end{proof}

This  gives a dynamical interpretation of the log discrepancies of
surface singularities. The lifted logarithmic dynamics along the exceptional
divisor produces the residue vector \(I\), and the intersection matrix of the
resolution transforms this vector into the log discrepancy vector
\[
(a_1+1,\ldots,a_r+1)^t.
\]
Thus, under the stated isolatedness assumptions, exceptional logarithmic
residues provide a dynamical test for log canonicity.

\begin{example}[The cyclic quotient singularity \(\frac{1}{m}(1,1)\)]
Let \(m\geq 2\), and let
$
Y_m:=\mathbb C^2/\mu_m,
$
where \(\mu_m\) acts by
$
\zeta\cdot(u,v)=(\zeta u,\zeta v).
$
Then \(Y_m\) has a cyclic quotient surface singularity of type
\(\frac{1}{m}(1,1)\) at the image of the origin.
Consider on \(\mathbb C^2\) the vector field
\[
w=u\frac{\partial}{\partial u}
-
v\frac{\partial}{\partial v}.
\]
It commutes with the scalar \(\mu_m\)-action, and therefore descends to a
holomorphic vector field \(v\) on \(Y_m\). Since \(w\) vanishes only at the
origin of \(\mathbb C^2\), the descended vector field \(v\) vanishes only at the
singular point of \(Y_m\). Thus the induced foliation \(\mathcal F\) is regular
on \((Y_m)_{\mathrm{reg}}\), and
$
\Sing(\mathcal F)\subset \Sing(Y_m).
$
The minimal resolution of \(Y_m\) is
$
\pi:X_m\to Y_m,
$
where \(X_m\) is the total space of the line bundle
\(\mathcal O_{\mathbb P^1}(-m)\). The exceptional divisor is the zero section
\[
D=E\simeq\mathbb P^1,
\qquad
E^2=-m.
\]
On the affine chart with coordinate
$
b=v/u
$
on \(\mathbb P^1\), set
$
x=u^m.
$
Then the quotient invariants are locally generated by
$
u^{m-k}v^k=x b^k,\ 
0\leq k\leq m,
$
and the exceptional curve is \(E=\{x=0\}\). In these coordinates, the lifted
vector field is
\[
\widetilde v
=
m x\frac{\partial}{\partial x}
-
2b\frac{\partial}{\partial b}.
\]
Hence \(\widetilde v\) is logarithmic along \(E\), since it is generated by
\(x\partial/\partial x\) in the normal logarithmic direction. Its restriction to
\(E\) is
\[
\widetilde v|_E
=
-2b\frac{\partial}{\partial b},
\]
which has one zero at \(b=0\).
On the second affine chart, with
$
c= (u/v)
$,$
y=v^m,
$
one obtains
\[
\widetilde v
=
-m y\frac{\partial}{\partial y}
+
2c\frac{\partial}{\partial c}.
\]
Thus the restriction to \(E\) has the second zero at \(c=0\), corresponding to
\(b=\infty\). Therefore the lifted logarithmic foliation has no divisorial
zeroes and has two isolated zeroes on the exceptional curve.

Let
\[
I_E:=
\sum_{p\in Z_E(\widetilde{\mathcal F})}
\widehat{\operatorname{Res}}^{\log}_{c_1,E}
(\widetilde{\mathcal F},E,p).
\]
We compute \(I_E\) directly. On the chart with coordinate \(b\), the induced
field on \(E\) is
$
-2b \partial/\partial b.
$
Thus the local residue at \(b=0\) is
\[
\operatorname{Res}_{b=0}
\left[
\frac{-2\,db}{-2b}
\right]
=
\operatorname{Res}_{b=0}\left[\frac{db}{b}\right]
=
1.
\]
On the chart with coordinate \(c\), the induced field is
$
2c\frac{\partial/\partial c},
$
and the local residue at \(c=0\) is
\[
\operatorname{Res}_{c=0}
\left[
\frac{2\,dc}{2c}
\right]
=
\operatorname{Res}_{c=0}\left[\frac{dc}{c}\right]
=
1.
\]
Therefore
$
I_E=2.$
The exceptional intersection matrix is
$
M=(E^2)=(-m).
$
Write
$
K_{X_m}=\pi^*K_{Y_m}+aE.
$
By adjunction,
$
(K_{X_m}+E)\cdot E=2g(E)-2=-2,
$
and since \(E^2=-m\), this gives
$
K_{X_m}\cdot E=-2-E^2=m-2.
$
On the other hand,
$$
K_{X_m}\cdot E
=
aE^2
=
-am.
$$
Hence
$
a=2/m-1.
$
Thus
$
b=a+1=2/m.
$
The identity of Theorem~\ref{thm:exceptional-residues-discrepancies} becomes
$
I=-Mb.
$
Here this reads
$
2=-(-m)(2/m),
$
which is true. Equivalently,
$
b=-M^{-1}I
=
-(-m)^{-1}\cdot 2
=
2/m.
$
Therefore the residue vector recovers
\[
a+1=\frac{2}{m},
\qquad
a=\frac{2}{m}-1.
\]
In particular,
\[
-M^{-1}I=\frac{2}{m}>0,
\]
so the criterion detects that \(Y_m\) is log terminal, hence log canonical. For
\(m=2\), one obtains \(a=0\), the canonical \(A_1\)-singularity. For \(m>2\),
one obtains
$
-1<a<0,$
so the singularity is log terminal but not canonical.
\end{example}

\subsection*{Acknowledgments}
MC is partially supported by the Universit\`a degli Studi di Bari and is a member of INdAM--GNSAGA.

\end{document}